\documentclass[11pt]{article}

\usepackage{amsmath, amsthm, amssymb, amscd}
\usepackage[numbers]{natbib}
\usepackage{graphicx}
\usepackage{bm}
\usepackage{tikz}
\usepackage[title]{appendix}

\usepackage{hyperref}
\usepackage{wrapfig}

\newcommand{\R}{\mathbb{R}}

\newtheorem{thm}{Theorem}
\newtheorem{lem}{Lemma}

\newtheorem{prop}{Proposition}
\newtheorem{cor}{Corollary}

\newcommand{\TV}{\operatorname{TV}}

\newcommand{\mH}{\mathcal{H}}
\newcommand{\mC}{\mathcal{C}}

\newcommand{\mA}{\mathcal{A}}
\newcommand{\mF}{\mathcal{F}}

\newcommand{\mE}{\mathcal{E}}

\begin{document}

\title{Asymptotic nonequivalence of density estimation and Gaussian white noise for small densities}

\author{Kolyan Ray\footnote{The research leading to these results has received funding from the European Research Council under ERC Grant Agreement 320637.\newline Email: \href{mailto:kolyan.ray@kcl.ac.uk}{kolyan.ray@kcl.ac.uk}, \href{mailto:schmidthieberaj@math.leidenuniv.nl}{schmidthieberaj@math.leidenuniv.nl}}\; and Johannes Schmidt-Hieber
 \vspace{0.1cm} \\
 {\em King's College London and Leiden University} }



\date{}
\maketitle

\begin{abstract}
\noindent It is well-known that density estimation on the unit interval is asymptotically equivalent to a Gaussian white noise experiment, provided the densities are sufficiently smooth and uniformly bounded away from zero. We show that a uniform lower bound, whose size we sharply characterize, is in general necessary for asymptotic equivalence to hold.
\end{abstract}

\paragraph{AMS 2010 Subject Classification:}
Primary 62B15; secondary 62G07, 62G10, 62G20.
%

\paragraph{Keywords:} Asymptotic equivalence; density estimation; Gaussian white noise model; small densities.

\section{Introduction}
\label{sec.intro}

A fundamental problem in nonparametric statistics is density estimation on a compact set, say the unit interval $[0,1]$, where we observe $n$ i.i.d. observations from an unknown probability density $f.$ If the parameter space $\Theta$ consists of densities $f$ that are uniformly bounded away from zero and have H\"older smoothness $\beta>1/2$, then a seminal result of Nussbaum \cite{nussbaum1996} establishes the global asymptotic equivalence of this experiment to the Gaussian white noise model where we observe $(Y_t)_{t\in [0,1]}$ arising from
\begin{align}
	dY_t = 2\sqrt{f(t)} dt + n^{-1/2} dW_t, \quad t\in [0,1], \ \ f\in \Theta,
	\label{eq.GWN_in_intro}
\end{align}
where $(W_t)_{t\in[0,1]}$ is a Brownian motion. The smoothness constraint is sharp: Brown and Zhang \cite{brown1998} construct a counterexample with a parameter space of H\"older smoothness exactly $\beta=1/2$ such that asymptotic equivalence does not hold. 

If two statistical experiments are asymptotically equivalent in the Le Cam sense, then asymptotic statements can be transferred between the experiments. More precisely, the existence of a decision procedure with risk $R_n$ for a given bounded loss function in one model implies the existence of a corresponding decision procedure with risk $R_n+o(1)$ in this loss in the other model. To derive asymptotic properties, one may therefore work in the simpler model and transfer the results to the more complex model. This is one of the main motivations behind the study of asymptotic equivalence. The last part of the introduction provides definitions and summarizes the concept of asymptotic equivalence of statistical experiments.

In practice, densities may be small or even zero on a subset of the domain, in which case the above result no longer applies. The goal of this article is to contribute to the general understanding of necessary conditions for asymptotic equivalence to hold, in particular the necessity of uniform boundedness away from zero. We show that without a minimal lower bound on the densities, density estimation and the Gaussian model \eqref{eq.GWN_in_intro} are always asymptotically nonequivalent, irrespective of the amount of H\"older smoothness.

In fact, we prove a more precise result by characterizing a size threshold such that if densities fall below this level, asymptotic equivalence never holds. In a companion paper \cite{raySH2016}, we show constructively that above this threshold, asymptotic equivalence may still hold. Our threshold is thus sharp in the sense that it is the smallest possible value a density can take such that asymptotic equivalence can hold.

We employ sample-size dependent parameter spaces $\Theta=\Theta_n$, as is typical in high-dimensional statistics. We prove that if the parameter spaces contain a sequence of $\beta$-smooth densities $(f_n)_n$ such that $\inf_{x\in [0,1]} f_n(x) =O(n^{-\beta/(\beta+1)})$ for all $n,$ as well as suitable neighbourhoods of $\beta$-smooth densities around the $(f_n)$, then the experiments are always asymptotically nonequivalent. This is a natural threshold for describing ``small" and ``large" densities with, for instance, different minimax rates attainable above and below this level \cite{patschkowski2016,ray2016IP}, see \eqref{eq.ptw_conv_rate} and the related discussion below.

From a practical perspective, Gaussian approximations have been proposed in density estimation (e.g. \cite{brown2009}) and one would like to better understand how ``large" a density must be for such methods to be applicable. The present work is a step in this direction. Furthermore, all the results presented in this paper also hold for the closely related case of Poisson intensity estimation, which is always asymptotically equivalent to density estimation, irrespective of density size or H\"older smoothness \cite{low2007,raySH2016}. This case is of particular practical relevance given the widespread use of Gaussian approximations for Poisson data \cite{hohage2016}, even for small intensities \cite{Makitalo2011}. We avoid further mention of Poisson intensity estimation for conciseness, but readers should bear in mind that all the present results and conclusions apply equally to that model.

There are few results establishing the necessity of conditions for asymptotic equivalence via counterexamples. For nonparametric regression, \cite{brown1998,efromovich1996} show the necessity of smoothness assumptions. The paper \cite{wang2002} establishes nonequivalence between the GARCH model and its diffusion limit under stochastic volatility, as well as their equivalence under deterministic volatility.

The proof we employ here relies on a reduction to binary experiments. The difficulty lies in both the construction of a suitable two-point testing problem and also in obtaining sufficiently good bounds on the total variation distance. Indeed, the situation is rather more subtle than one might first imagine. For two-point hypothesis testing problems, we show that one can consistently test between the alternatives in one model if and only if one can do so in the other (Lemma \ref{lem.inform_distances}). To establish nonequivalence, one must therefore construct alternatives which can be separated with a positive probability that is strictly bounded away from zero and one and for which suitable bounds can be computed.

Although for small signals, density estimation and the Gaussian white noise model \eqref{eq.GWN_in_intro} are no longer asymptotically equivalent, many aspects of their statistical theory remain the same. As mentioned above, simple hypothesis testing is essentially the same in both models without any lower bound on the densities. To explain this in more detail, suppose that $g_n$ and $h_n$ are two sequences of densities and denote the probability measures in the density estimation model and the Gaussian model  \eqref{eq.GWN_in_intro} by $P_f^n$ and $Q_f^n$ respectively. The sums of the type I and II error probabilities of the Neyman-Pearson test for the simple hypotheses
$$H_0: f=g_n  \quad \quad H_1: f=h_n$$ in the two models are $\tfrac 12(1-\|P_{g_n}^n -P_{h_n}^n\|_{\TV})$ and $\tfrac 12(1-\|Q_{g_n}^n -Q_{h_n}^n\|_{\TV})$ respectively. By Lemma \ref{lem.inform_distances} below, $\|P_{g_n}^n -P_{h_n}^n\|_{\TV} \rightarrow 1$ if and only if $\|Q_{g_n}^n -Q_{h_n}^n\|_{\TV} \rightarrow 1,$ which shows that we can consistently test against a simple alternative in one model if and only if we can do so in the other model. This argument requires no lower bound on the densities. The Hellinger distance also behaves very similarly in the two models, see Lemma \ref{lem.inform_distances} for a precise statement. It is an interesting phenomenon that while the models are potentially far apart with respect to the Le Cam distance, information distances, such as the total variation and Hellinger distance, remain close. Although this does not hold for all common information measures, for instance the Kullback-Leibler divergence, it nevertheless suggests that negative results for small densities in the Le Cam sense may be misleading, since many important statistical properties still carry over between models.

Beyond density estimation, uniform boundedness away from zero is a standard assumption in the asymptotic equivalence literature \cite{brown2004,genon-catalot2002,golubev2010,grama1998,nussbaum1996}. However, this assumption is not always required, including in regression type models \cite{brown1996,SH2014} and even some non-linear problems, such as diffusion processes \cite{dalalyan2006,dalalyan2007,delattre2002}. A better understanding of the necessity of such conditions is therefore of interest in a wide variety of models.

\section{Main results}
\label{sec.main}

\subsection*{Basic notation and definitions}

For two functions $f,g$ on $[0,1]$, we write $f\leq g$ if $f(x) \leq g(x)$ for all $x\in[0,1]$ and let $\|f\|_2$ denote the $L^2$-norm of $f$. Given two probability measures $P,Q$ with densities $p,q$ with respect to some dominating measure $\nu$, we recall the total variation distance $\|P-Q\|_{\TV} := \tfrac 12 \int |p-q| d\nu$ and Hellinger distance $H(P,Q) := (\int (\sqrt{p}-\sqrt{q})^2 d\nu)^{1/2}$. 

A statistical experiment $\mE(\Theta)=(\Omega, \mA, (P_\theta: \theta \in \Theta))$ consists of a sample space $\Omega$ with associated $\sigma$-algebra $\mA$ and a family $(P_\theta: \theta \in \Theta)$ of probability measures all defined on the measurable space $(\Omega,\mA)$. We call $\mE(\Theta)$ dominated if there exists a probability measure $\mu$ such that any $P_\theta$ is dominated by $\mu.$ Furthermore, $\mE(\Theta)$ is said to be Polish if $\Omega$ is a Polish space and $\mA$ is the associated Borel $\sigma$-algebra. If $\mE(\Theta)=(\Omega, \mA, (P_\theta: \theta \in \Theta))$ and $\mF(\Theta)=(\Omega', \mA', (Q_\theta: \theta \in \Theta))$ are two Polish and dominated experiments indexed by the same parameter space, the Le Cam deficiency can be defined as 
\begin{align*}
	\delta\big(\mE(\Theta) , \mF(\Theta)\big) := \inf_M \sup_{\theta\in \Theta} \big\| MP_\theta^n - Q_\theta^n \big\|_{\TV},
\end{align*} 
where the infimum is taken over all Markov kernels $M:\Omega_1\times \mA_2 \rightarrow [0,1]$. The Le Cam distance is defined as
\begin{align*}
	\Delta\big(\mE(\Theta) , \mF(\Theta)\big) := \max \big\{ \delta\big(\mE(\Theta) , \mF(\Theta)\big) , \delta\big(\mF(\Theta) , \mE(\Theta)\big) \big\},
\end{align*}
which defines a pseudo-distance on the space of all experiments with parameter space $\Theta.$ One may generalize the definition of Le Cam deficiency to spaces that are neither Polish nor dominated upon replacing the notion of Markov kernel with a more general transition \cite{LeCam1986,str}. However, we refrain from doing so here since these notions coincide in the Polish and dominated experiments we consider in this article, see (68) and Proposition 9.2 of \cite{nussbaum1996}. Finally, we say that two sequences of experiments $\mE_n(\Theta_n)=(\Omega_n, \mA_n, (P_\theta^n: \theta \in \Theta_n))$ and $\mF(\Theta_n)=(\Omega_n', \mA_n', (Q_\theta^n: \theta \in \Theta_n))$ are asymptotically equivalent if $\Delta(\mE_n(\Theta_n),\mF_n(\Theta_n)) \rightarrow 0$ as $n\rightarrow \infty$. General treatments on asymptotic equivalence can be found in \cite{LeCam1986,str}.


In this article, we consider the following two statistical experiments.

{\it Density estimation $\mE_n^D(\Theta)$:} In nonparametric density estimation, we observe $n$ i.i.d. copies $X_1,\ldots,X_n$ of a random variable on $[0,1]$ with unknown Lebesgue density $f.$ The corresponding statistical experiment is $\mE_n^D(\Theta)=([0,1]^n ,\sigma([0,1]^n), (P_f^n : f\in \Theta))$ with $P_f^n$ the product probability measure of $X_1, \ldots,X_n.$

{\it Gaussian white noise experiment $\mE_n^G(\Theta)$:} We observe the Gaussian process $(Y_t)_{t\in [0,1]}$ arising from \eqref{eq.GWN_in_intro} with $f\in \Theta$ unknown. Denote by $\mC([0,1])$ the space of continuous functions on $[0,1]$ and let $\sigma(\mC([0,1]))$ be the $\sigma$-algebra generated by the open sets with respect to the uniform norm. The Gaussian white noise experiment is then given by $\mE_n^G(\Theta)=(\mC([0,1]) ,\sigma(\mC([0,1])), (Q_f^n : f\in \Theta))$ with $Q_f^n$ the distribution of $(Y_t)_{t\in [0,1]}.$

\subsection*{Function spaces}

Denote by $\lfloor \beta \rfloor$ the largest integer strictly smaller than $\beta.$ The usual H\"older semi-norm is given by $|f|_{\mC^\beta} := \sup_{x\neq y, x,y\in [0,1]} |f^{(\lfloor \beta \rfloor)}(x) - f^{(\lfloor \beta \rfloor)}(y)| /|x-y|^{\beta - \lfloor \beta \rfloor}$ and the H\"older norm is $\| f \|_{\mC^\beta} := \| f\|_\infty + \| f^{(\lfloor\beta\rfloor)} \|_\infty + |f|_{\mC^\beta}.$ Consider the space of $\beta$-smooth H\"older densities with H\"older norm bounded by $R,$
\begin{equation*}
\mC^\beta(R) := \big\{ f: [0,1] \rightarrow \R \  :  \  f\geq 0, \ \int_0^1 f(u) du =1, \  f^{(\lfloor \beta \rfloor)} \text{ exists}, \  \| f \|_{\mC^\beta} \leq R \big\}. 
\end{equation*}
If $0<\beta \leq 2,$ the pointwise rate of estimation at any $x \in (0,1)$ over the parameter space $\mC^\beta(R)$ is given by
\begin{equation}\label{eq.ptw_conv_rate}
n^{-\frac{\beta}{\beta+1}} + \left( \frac{f(x)}{n}\right)^{\frac{\beta}{2\beta+1}},
\end{equation}
with upper and lower bounds matching up to $\log n$ factors (see Theorems 3.1 and 3.3 of \cite{patschkowski2016} for density estimation and Theorems 1 and 2 of \cite{ray2016IP} for the Gaussian white noise model). There is thus a phase transition in the estimation rate for small densities occurring at the $n$-dependent signal size $f(x) \asymp n^{-\frac{\beta}{\beta+1}}$. This is the same boundary for asymptotic nonequivalence proved in Theorem \ref{thm.lb_small} below, so that in some respects at least, the two experiments do behave differently from one another below this threshold. However, despite asymptotic nonequivalence, many other properties, such as minimax rates and consistent testing, are still asymptotically the same below this threshold. Indeed, the counterexample we construct lies right on the boundary of testing problems and in some sense only narrowly fails. The importance of the threshold $f(x) \asymp n^{-\frac{\beta}{\beta+1}}$ is not isolated to minimax estimation rates and asymptotic equivalence and seems to play a fundamental role for small densities, for example being necessary to obtain sharp rates when estimating the support of a density \cite{patschkowski2016}. For further discussion see \cite{patschkowski2016,ray2016IP}.

The rate of convergence \eqref{eq.ptw_conv_rate} does not extend to $\beta>2$ using the usual definition of H\"older smoothness due to the existence of functions which are highly oscillatory near zero (Theorem 3 of \cite{ray2016IP}). A natural way to attain the rate for smoothness $\beta>2$ is to impose a shape constraint ruling out such pathological behaviour. On $\mC^\beta$, define the flatness seminorm
\begin{align}\label{eq.flat_def}
|f |_{\mathcal{H}^\beta} = \max_{1 \leq j <\beta } \| |f^{(j)}|^\beta/|f|^{\beta-j} \|_\infty^{1/j} = \max_{1\leq j <\beta} \left( \sup_{x\in[0,1]} \frac{|f^{(j)}(x)|^\beta}{|f(x)|^{\beta-j}} \right)^{1/j}
\end{align} 
with $0/0$ defined as $0$ and $|f|_{\mH^\beta}=0$ if $0<\beta \leq 1.$ The quantity $| f |_{\mathcal{H}^\beta}$ measures the flatness of a function near zero in the sense that if $f(x)$ is small, then the derivatives of $f$ must also be small in a neighborhood of $x$. Define $\| f\|_{\mH^\beta} := \|f\|_{\mC^\beta} + |f|_{\mH^\beta}$ and consider the space of densities
\begin{align*}
	\mH^\beta(R) := \{ f\in \mC^\beta(R) \ : \  \|f\|_{\mH^\beta}\leq R\}.
\end{align*}
Notice that $\mH^\beta(R)=\mC^\beta(R)$ for $\beta\leq 1.$ For further discussion and properties of the function space $\mH^\beta(R)$, see \cite{RaySchmidt-Hieber2015c}.

The reason we construct a counterexample in $\mH^\beta(R)$ is to concretely show that asymptotic nonequivalence is not due to functions that are highly oscillatory near zero, but also holds for typical H\"older functions. Thus even when considering only ``nice" H\"older functions, for which the rate \eqref{eq.ptw_conv_rate} is attainable, nonequivalence still holds.

\subsection*{Asymptotic nonequivalence}

To obtain suitable lower bounds on the Le Cam deficiencies, we require that the small densities are not isolated in the parameter space $\Theta_n$, meaning we must introduce a notion of interior parameter space. This is in some sense necessary, since asymptotic equivalence may still hold when the small density behaviour is driven by a parametric component, in particular having finite Hellinger metric dimension. For further discussion on this point, see Proposition \ref{prop.parametric_AE} below.

The following result is the main contribution of this article, showing that if $$\inf_{f \in \Theta_n} \inf_{x\in[0,1]} f(x) \lesssim n^{-\beta/(\beta+1)},$$ then the Le Cam deficiency is bounded from below by a positive constant for sufficiently large $n.$ In this case, the experiments are asymptotically nonequivalent.

\begin{thm}
\label{thm.lb_small}
Let $R, \beta >0.$  There exists a constant $c>1,$ not depending on $R,$ such that if $(f_{0,n})_n \subset \Theta_n \cap \mH^\beta(R)$ is a sequence satisfying $\inf_{x\in[0,1]} f_{0,n}(x) \leq \tfrac 12 R^{1/(\beta+1)}n^{-\beta/(\beta+1)}$ and  $\{ f\in \mH^\beta(cR) : c^{-1} f_{0,n} \leq f \leq cf_{0,n}\} \subset \Theta_n$ for all $n\geq 2$, then
\begin{align*}
	\delta\big(\mE_n^D(\Theta_n), \mE_n^G(\Theta_n)\big) \geq 0.007 + o(1) >0.
\end{align*}
\end{thm}

The assumption is that the parameter space $\Theta_n$ is rich enough to contain a function $f_{0,n}$ that somewhere falls below the threshold $n^{-\beta/(\beta+1)}$, together with all the functions in $\mH^\beta(cR)$ lying in the band $x\mapsto [c^{-1}f_{0,n}(x), cf_{0,n}(x)]$ around $f_{0,n}.$ An explicit expression for $c$ can be obtained from the proof, see \eqref{eq:c}. As a particular example, the norm balls $\Theta_n = \mH^\beta(R)$ satisfy the above assumptions.

\begin{cor}
\label{cor}
For any $\beta>0$ and sufficiently large $R\geq R_0(\beta),$
\begin{align*}
	\Delta\big(\mE_n^D(\mH^\beta(R)), \mE_n^G(\mH^\beta(R))\big) \geq 0.007 + o(1) >0. 
\end{align*}
\end{cor}

\begin{proof}[Proof of Corollary \ref{cor}]
For $\beta>0,$ consider the density $f(x) =(\beta+1)x^\beta.$ For any integer $0 \leq j <\beta$, $f^{(j)}(x) = [\Gamma(\beta+2)/\Gamma(\beta-j+1)] x^{\beta-j}$, where $\Gamma(t)$ denotes the Gamma function. This implies that $\|f\|_\infty+\|f^{(\lfloor \beta \rfloor)}\|_\infty \leq (\beta+1) + \Gamma(\beta+2)/\Gamma(\beta-\lfloor\beta\rfloor+1)$ and $|f|_{\mH^\beta} = \max_{1\leq j <\beta} \Gamma(\beta+2)^{\beta/j}\Gamma(\beta-j+1)^{-\beta/j}(\beta+1)^{-(\beta-j)/j}$. For any $z\in [0,1]$ and $\gamma\in[0,1]$, $1\leq z^\gamma +(1-z)^\gamma$, which implies $|x^\gamma-y^\gamma|\leq |x-y|^\gamma$ for $x,y \geq 0.$ Hence, $|f|_{\mC^\beta} \leq \Gamma(\beta+2)/\Gamma(\beta-\lfloor\beta\rfloor+1)$, so that $\|f\|_{\mH^\beta} \leq C_\beta$ for some finite constant $C_\beta$ depending only on $\beta.$ Let $c>1$ be the constant in Theorem \ref{thm.lb_small}. If $R\geq C_\beta c,$ we may apply Theorem \ref{thm.lb_small} with the constant sequence $f_{0,n}= f,$ since then $(f_{0,n})_n \subset \mH^\beta(C_\beta) \subset \mH^\beta(R/c) .$ By Theorem \ref{thm.lb_small} with $\Theta_n=\mH^\beta(R)$ and $R$ replaced by $R/c,$ the assertion follows.
\end{proof}

Since $\|f\|_{\mH^\beta}\geq \|f\|_\infty \geq 1$ for any density $f$ on $[0,1],$ the radius $R$ in the previous corollary must be larger than some $R_0(\beta)$, otherwise the parameter space $\mH^\beta(R)$ is empty. For small densities, the Gaussian white noise model \eqref{eq.GWN_in_intro} can be asymptotically more informative than density estimation. This result is only interesting in the case $\beta> 1/2,$ since for $\beta \leq 1/2,$ asymptotic equivalence can fail even if all densities are uniformly bounded away from zero \cite{brown1998}.


Under general conditions, if $\Theta_n \subset \mH^\beta(R)$ for $1/2<\beta \leq 1$ and $\inf_{f\in \Theta_n}\inf_{x\in[0,1]} f(x)  \gg n^{-\frac{\beta}{\beta+1}} \log^8 n$, the squared Le Cam deficiencies between density estimation and the Gaussian model \eqref{eq.GWN_in_intro} are exactly of the order 
\begin{align}
	\min \left\{ 1, n^{\frac{1-2\beta}{2\beta+1}} \sup_{f\in \Theta_n}\int_0^1 f(x) ^{-\frac{2\beta+3}{2\beta+1}} dx \right\},
	\label{eq.rate_in_intro}
\end{align}
see Theorem 4 of \cite{raySH2016}. In particular, if $f$ is uniformly bounded away from zero we recover the rate $\min \{ 1,n^{(1-2\beta)/(2\beta+1)}\}$, so that the experiments are asymptotically equivalent if and only if $\beta>1/2$. As we now show by example, in view of \eqref{eq.rate_in_intro}, the threshold $n^{-\beta/(\beta+1)}$ obtained in Theorem \ref{thm.lb_small} is essentially sharp up to a logarithmic factor.

Consider the densities $f_{0,n}(x) \propto x^\beta+n^{-\frac{\beta}{\beta+1}}M_n$ with $M_n \gtrsim \log^8 n$ diverging, which satisfy $\inf_{x\in[0,1]} f_{0,n}(x) \propto n^{-\frac{\beta}{\beta+1}} M_n \gg n^{-\frac{\beta}{\beta+1}}$ and $f_{0,n} \in \mH^\beta(R)$ for $R>0$ large enough. For  $c>0$ the constant from Theorem \ref{thm.lb_small}, set
\begin{align*}
\Theta_n = \big\{ f\in \mH^\beta(cR) : c^{-1} f_{0,n} \leq f \leq cf_{0,n} \big\}.
\end{align*}
Since $\inf_{f\in \Theta_n} \inf_{x\in[0,1]}f(x) \gtrsim n^{-\frac{\beta}{\beta+1}} \log^8 n$, applying \eqref{eq.rate_in_intro},
\begin{align*}
\Delta\big(\mE_n^D(\Theta_n),\mE_n^G(\Theta_n)\big)^2 \asymp n^{\frac{1-2\beta}{2\beta+1}}\int_0^1 f_{0,n}(x)^{-\frac{2\beta+3}{2\beta+1}} dx \asymp M_n^\frac{(1-2\beta)(\beta+1)}{\beta(2\beta+1)} \rightarrow 0
\end{align*}
for $1/2<\beta\leq 1$, so that density estimation and the Gaussian model \eqref{eq.GWN_in_intro} with parameter spaces $\Theta_n$ are asymptotically equivalent. In summary, asymptotic equivalence always fails below the threshold $n^{-\frac{\beta}{\beta+1}}$, but may still hold for any level larger than $n^{-\frac{\beta}{\beta+1}} \log^8 n$, thereby showing that Theorem \ref{thm.lb_small} is sharp up to a logarithmic factor. The $\log^8 n$ factor is a technical artifact arising from the proof of \eqref{eq.rate_in_intro}.

\subsection*{Asymptotic equivalence for small densities in parametric settings}

Asymptotic nonequivalence due to small densities is a feature of fully nonparametric models and our conclusions do not necessarily apply in parametric models. We illustrate this via an example, whose proof we defer to the end of the article.

\begin{prop}\label{prop.parametric_AE}
Consider the probability density $g(x) = 960x^2(1/2-x)^2 1_{[0,1/2]}(x)$. For $K\subset (0,1/2)$ a compact interval, consider the location family $\Theta(K) = \{ f_\theta(x) = g(x-\theta): \theta\in K \}$. For this parameter space, density estimation and the Gaussian model \eqref{eq.GWN_in_intro} are asymptotically equivalent, that is as $n\rightarrow \infty$,
\begin{align*}
\Delta \big(\mE_n^D(\Theta(K)),\mE_n^G(\Theta(K))\big) \rightarrow 0.
\end{align*}
\end{prop}

The densities in the location family $\Theta(K)$ are not bounded away from zero on $[0,1]$, with $f_\theta$ equal to zero on $[0,\theta]\cup [\theta+1/2,1]$, yet asymptotic equivalence still holds. The reason for this is that for areas of $[0,1]$ where there are too few observations to admit a Gaussian approximation, the required information is provided by the parameter estimates for $\theta$. A sufficient condition for this is finite Hellinger \textit{metric dimension} in the density model, not to be confused with finite vectorial dimension of the parameter space, see Assumption (A3) of Le Cam \cite{LeCam1985}. Recall that a family $(f_\theta:\theta \in \Theta')$ of density functions is said to have finite Hellinger metric dimension if there exists a number $D\geq 0$ such that every subset of $(f_\theta:\theta \in \Theta')$ which can be covered by an $\varepsilon$-ball in Hellinger distance $H$, can be covered by at most $2^D$ $\varepsilon/2$-balls in $H$, where $D$ does not depend on $\varepsilon$. For example, the family of densities $\{ C\exp(-|x-\theta|^\alpha) :\theta\in\R\}$ on $\R$ for some $\alpha\in(0,1/2)$ has vectorial dimension one yet does not have finite Hellinger metric dimension, see Remark 2 after Theorem 4.3 of Le Cam \cite{LeCam1985}. In this sense, Theorem \ref{thm.lb_small} is truly a nonparametric result.

One can extend this further by considering parameter spaces with a-priori known zeroes. For instance Mariucci \cite{mariucci2016b} establishes asymptotic equivalence for densities of the form $fg$, where $f\in C^{\beta}$, $\beta>1$, is an unknown function uniformly bounded away from zero and $g$ is a given known function that is possibly small. In view of the above, one may interpret this as a form of semiparametric model, with a parametric part $g$ determining the density for small values and the nonparametric part $f$ doing so for large values. Thus for areas of $[0,1]$ with sufficient observations, one can fit a Gaussian approximation based on $f$ as usual, whereas for regions with insufficient observations, one must use the information provided by the parameter estimate for $g$, which in this particular example arises from a zero-dimensional family since $g$ is known exactly.

\subsection*{Overview of the proof}

The proof of Theorem \ref{thm.lb_small} is based on a reduction to binary experiments and a direct comparison of the total variation distances between the parameters using the following lemma.

\begin{lem}
\label{lem.lecam_lb}
 Let $\mE_1^b=(\Omega_1,\mathcal{A}_1,(P_{1,i}: i\in \{1,2\}))$ and $\mE_2^b=(\Omega_2,\mathcal{A}_2,(P_{2,i}: i\in \{1,2\}))$ be binary experiments. Then
\begin{align*}
  \delta(\mE_1^b,\mE_2^b)
  \geq \frac 12 \big( \|P_{2,1}-P_{2,2}\|_{\TV} - \|P_{1,1}-P_{1,2}\|_{\TV}\big).
\end{align*}
\end{lem}

\begin{proof}
We have the explicit formula $\delta(\mE_1^b,\mE_2^b)=\sup_{0\leq \xi \leq 1} [g_1(\xi)-g_2(\xi)]$ with $g_j(\xi)=\inf [(1-\xi)P_{j,1}\phi+\xi P_{j,2}(1-\phi)]$ the error function in $\mE_j^b,$ $j\in\{1,2\}$, and where the infimum is over all tests $\phi$, see Strasser \cite{str}, Corollary 15.7 and Definition 14.1. Notice that the definition of deficiency in \cite{str}, Definition 15.1, has an additional factor $1/2$. The result then follows with $g_j(1/2)=\tfrac 12 (1-\|P_{j,1}-P_{j,2}\|_{\TV})$ (\cite{str}, p. 71).
\end{proof}

To establish asymptotic nonequivalence for a discrete experiment and its continuous analogue, a standard approach is to consider a sequence of binary experiments such that the total variation distance in the discrete model is zero (i.e. both measures are the same) but the total variation distance in the continuous model is positive. Lemma \ref{lem.lecam_lb} then yields asymptotic nonequivalence. 

This approach cannot be used here and the proof of Theorem \ref{thm.lb_small} requires a much more careful choice of the sequence of binary experiments. Consider a sequence of binary experiments $\{P_{f_n}^n,P_{g_n}^n\}$ in the density estimation setting with corresponding binary experiments $\{Q_{f_n}^n, Q_{g_n}^n\}$ in the Gaussian white noise model. The following result shows that the total variation distance in one experiment tends to zero if and only if the total variation in the other experiment also tends to zero. The same holds if the total variation distances both tend to one. Thus, in order to construct a lower bound via Lemma \ref{lem.lecam_lb}, such sequences cannot be used.

\begin{lem}
\label{lem.inform_distances}
Let $(f_n)_n$ and $(g_n)_n$ be arbitrary sequences of densities in both experiments $\mE_n^D(\Theta)$ and $\mE_n^G(\Theta).$ For $P_f^n$ the product probability measure for density estimation and $Q_f^n$ the law of the Gaussian white noise model \eqref{eq.GWN_in_intro},
\begin{align}
	\|P_{f_n}^n-P_{g_n}^n\|_{\TV}\rightarrow 0 \quad \Leftrightarrow \quad \|Q_{f_n}^n-Q_{g_n}^n\|_{\TV}\rightarrow 0 \quad 
	\Leftrightarrow \quad n \int (\sqrt{f_n}-\sqrt{g_n})^2 \rightarrow 0
	\label{eq.inform_dist_1}
\end{align}
and
\begin{align}
	\|P_{f_n}^n-P_{g_n}^n\|_{\TV}\rightarrow 1 \quad \Leftrightarrow \quad \|Q_{f_n}^n-Q_{g_n}^n\|_{\TV}\rightarrow 1 \quad 
	\Leftrightarrow \quad n \int (\sqrt{f_n}-\sqrt{g_n})^2 \rightarrow \infty.
	\label{eq.inform_dist_2}
\end{align}
If $H^2(P,Q) = \int(\sqrt{dP}-\sqrt{dQ})^2$ denotes the Hellinger distance, then for $n>1,$
\begin{align}
	H^2\big(Q_{f_n}^n,Q_{g_n}^n\big) 
	\leq H^2\big(P_{f_n}^n,P_{g_n}^n\big) 
	\leq H^2\big(Q_{f_n}^n,Q_{g_n}^n\big)  +\frac{2\log n}{n}.
	\label{eq.hellinger_info_approx}
\end{align}
\end{lem}

\begin{proof}
We first prove \eqref{eq.hellinger_info_approx}. By Lemma \ref{lem.bds_of_info_distances} below, $H^2(Q_{f_n}^n, Q_{g_n}^n)  = 2-2\exp(- \tfrac n2 \|\sqrt{f_n}-\sqrt{g_n}\|_2^2).$ Together with Lemmas 2.17 and 2.19 of \cite{str}, this proves 
\begin{align*}
	H^2\big(Q_{f_n}^n,Q_{g_n}^n\big) 
	\leq H^2\big(P_{f_n}^n,P_{g_n}^n\big) 
	\leq H^2\big(Q_{f_n}^n,Q_{g_n}^n\big)  +\frac 12 \int (\sqrt{f_n}-\sqrt{g_n})^2.	
\end{align*}
Distinguishing whether the term $\tfrac n2 \int (\sqrt{f_n}-\sqrt{g_n})^2$ is larger or smaller than $\log n,$ and using that $H^2(Q_{f_n}^n,Q_{g_n}^n)\geq 2-2n^{-1}$ if it is, then establishes \eqref{eq.hellinger_info_approx}.

To verify the first two assertions of the lemma, notice that by Le Cam's inequalities (Lemma 2.3 in \cite{tsybakov2009}), for any probability measures $P,Q,$ 
\begin{align}
	\frac{1}{2} H^2(P,Q) \leq  \|P-Q\|_{\TV} \leq  \min \left\{ H(P,Q) , \Big( 1 - \frac{1}{2}\big(1- \tfrac 12 H^2(P,Q) \big)^2 \Big) \right\}.
	\label{eq.Hellinger_TV_ineqs}
\end{align}
Consequently, the total variation of two sequences $(P_n)$ and $(Q_n)$ converges to zero if and only if $H^2(P_n,Q_n) \rightarrow 0.$ Similarly,  $\|P_n-Q_n\|_{\TV}\rightarrow 1$ if and only if $H^2(P_n,Q_n) \rightarrow 2.$ Using \eqref{eq.hellinger_info_approx} and $H^2(Q_{f_n}^n, Q_{g_n}^n) =  2-2\exp(- \tfrac n2 \|\sqrt{f_n}-\sqrt{g_n}\|_2^2),$ \eqref{eq.inform_dist_1} and \eqref{eq.inform_dist_2} follow.
\end{proof}

In view of this, we must construct sequences such that the total variation distances in the two experiments tend neither to zero nor one and are separated for $n$ large enough.

In the following we describe the ideas that finally lead to a lower bound. As a first step, we use Lemma \ref{lem.TV_bd} below to show that in the density estimation model,
\begin{align*}
	\| P_f^n -P_g^n \|_{\TV} \leq 1- \Big( 1- \frac{\|f-g\|_{1}}{2}\Big)^n.
\end{align*}
In the Gaussian white noise model, we have by Lemma \ref{lem.bds_of_info_distances} that $\|Q_f^n- Q_g^n\|_{\TV} = 1-2\Phi(-\sqrt{n} \|\sqrt{f}-\sqrt{g}\|_2)$ with $\Phi$ the distribution function of a standard normal random variable. For the Le Cam deficiency of the binary experiments with parameter space $\Theta=\{f,g\},$ Lemma \ref{lem.lecam_lb} then implies the following lower bound:
\begin{align}
	\delta\big( \mE_n^D(\{f,g\}), \mE_n^G(\{f,g\}) \big) 
	\geq \frac 12 \Big( 1- \frac{\|f-g\|_{1}}{2}\Big)^n - \Phi\big(-\sqrt{n} \|\sqrt{f}-\sqrt{g}\|_2\big).
	\label{eq.lb_binary_explicit}
\end{align}
To prove asymptotic nonequivalence, we therefore want to construct sequences $(f_n)_n, (g_n)_n \subseteq \Theta$ such that the total variation $\|f_n-g_n\|_{1}$ is small while the Hellinger distance $\|\sqrt{f_n}-\sqrt{g_n}\|_2$ is large. The largest value of the Hellinger distance is given by Le Cam's inequalities \eqref{eq.Hellinger_TV_ineqs}, $\|\sqrt{f_n}-\sqrt{g_n}\|_2 \leq \|f_n-g_n\|_1^{1/2}.$ An inspection of the proof shows that equality is achieved if for all $x,$ either $f_n(x)=0,$ $g_n(x) =0$ or $f_n(x)=g_n(x).$ This is a first indication that the bound \eqref{eq.lb_binary_explicit} is particularly useful for small densities. 

We now provide a heuristic showing that the reduction to a binary experiment can only be used if the parameter space contains small densities. Observe that in view of Lemma \ref{lem.inform_distances}, we need $\|f_n-g_n\|_{1} \asymp \|\sqrt{f_n}-\sqrt{g_n}\|_2^2 \asymp 1/n$ to show that the Le Cam deficiency is lower bounded by a positive constant. The standard approach for nonparametric two hypothesis lower bounds is to consider one function as a local perturbation of the other. For fixed $f_n$ and $K \not\equiv 0$ a smooth function with $\int K=0$ and support in $[-1,1],$ set
\begin{align}
	g_n = f_n+ h_n^\beta K\Big( \frac{\cdot - x_0}{h_n} \Big),
	\label{eq.density_perturb}
\end{align}
where $x_0 \in (0,1)$ is fixed and $h_n>0$. If $f_n$ is a $\beta$-H\"older smooth function, a standard argument shows that $g_n$ is also $\beta$-H\"older smooth. If the perturbation is small enough, then $g_n\geq 0$ is a density since $\int K=0$. The perturbation has height $O(h_n^\beta)$ and support of length $2h_n,$ which means that the total variation distance is of the order $h_n^{\beta+1}.$ To ensure that $\|f_n-g_n\|_{1} \asymp 1/n,$ we therefore take $h_n \asymp n^{-1/(\beta+1)}.$ On the other hand, the squared Hellinger distance satisfies
\begin{align}
	\|\sqrt{f_n}-\sqrt{g_n}\|_2^2 = \int \frac{(f_n -g_n)^2}{(\sqrt{f_n} +\sqrt{g_n})^2} \asymp \frac{h_n^{2\beta+1}}{f_n(x_0) + h_n^\beta}.
	\label{eq.why_it_doesnt_work_for_large_dens}
\end{align} 
To ensure the right hand side is of order $1/n,$ we consequently need $f(x_0) =O(h_n^\beta) = O(n^{-\frac{\beta}{\beta+1}}).$ If all densities are bounded away from zero, a different approach based on a multiple testing problem is needed to obtain sharp lower bounds \cite{raySH2016}.

To summarize, we have used that in the density estimation model the total variation of the product measures $P_f^n$ and $P_g^n$ can be bounded in terms of the total variation between the densities $f$ and $g.$ On the contrary, in the Gaussian white noise model the total variation distance is a function of the Hellinger distance of $f$ and $g.$ The total variation distance is bounded from below by the {\it  squared} Hellinger distance and from above by the Hellinger distance via \eqref{eq.Hellinger_TV_ineqs}. Nonequivalence can therefore be established using the inequality \eqref{eq.lb_binary_explicit} if the total variation between $f$ and $g$ behaves like the squared Hellinger distance, which happens exactly when the densities are small.

\section{Proofs}
\label{sec.lbs_irregular}

We construct two test functions and use that the Le Cam deficiency is bounded from below by the difference of the total variation distances. To prove Theorem \ref{thm.lb_small}, it is by \eqref{eq.lb_binary_explicit} enough to show that for some densities $f_{1,n}, f_{2,n}\in \Theta_n,$
\begin{align}
	\frac 12 \Big( 1- \frac{\|f_{1,n}-f_{2,n}\|_{1}}{2}\Big)^n - \Phi\big(-\sqrt{n} \|\sqrt{f_{1,n}}-\sqrt{f_{2,n}}\|_2\big)
	& \geq \frac 12e^{-\frac{3}{2}}\Big( 1-\sqrt{\frac e{\pi}}\Big) +o(1) \notag \\
	&\geq 0.007+ o(1).
	\label{eq.low_bound_to_show}
\end{align}
We henceforth omit the index $n$ for convenience, writing $f_1=f_{1,n}$ and $f_2=f_{2,n}.$ Before we describe the construction of $f_1, f_2,$ we first recall the following basic property of functions in the flat H\"older space $\mH^\beta.$

\begin{lem}[Lemma 1 in \cite{RaySchmidt-Hieber2015c}]
\label{lem.fx_local_bd}
Suppose that $f \in \mH^\beta$ with $\beta>0$ and let $a = a(\beta)>0$ be any constant satisfying $(e^a-1) + a^\beta / (\lfloor \beta \rfloor!) \leq 1/2.$ Then for 
\begin{align*}
	|h| \leq a \left( \frac{|f(x)|}{\|f\|_{\mH^\beta}}\right)^{1/\beta},
\end{align*}
we have
\begin{equation*}
|f(x+h) - f(x)| \leq \frac{1}{2} |f(x)| ,
\end{equation*}
implying in particular, $|f(x)|/2 \leq |f(x+h)| \leq 3|f(x)|/2$.
\end{lem}

{\it Construction of $f_1, f_2 \in \Theta_n :$} For a given density $f_0,$ we consider two perturbations of $f_0$ for an $x_0$ such that $f_0(x_0) \asymp n^{-\beta/(\beta+1)}.$ This choice is natural in view of \eqref{eq.why_it_doesnt_work_for_large_dens}. The way we construct the perturbations is for technical convenience slightly different than in \eqref{eq.density_perturb}. By assumption there exist densities $f_0:=f_{0,n} \in \mH^\beta( R)$ such that for some $x_0:=x_{0,n}\in [0,1],$ $ R^{1/(\beta+1)}(2n)^{-\frac{\beta}{\beta+1}} \leq f_0(x_0)\leq R^{1/(\beta+1)}n^{-\frac{\beta}{\beta+1}}$ for all $n.$ Without loss of generality, we may assume that $x_0\leq 1/2.$ We must ensure that we can apply Lemma \ref{lem.fx_local_bd} on the support of the perturbations, which motivates the following definitions.  With $0<a\leq 1/4$ a solution of $e^a-1+a^\beta/\lfloor \beta \rfloor !\leq 1/2,$ set
\begin{align*}
	F:= \frac{a f_0(x_0)^{\frac{\beta+1}{\beta}}}{4R^{\frac 1{\beta}}} \leq \frac{1}{16n}
\end{align*}
and observe that $F\geq a/(8n).$ Given $x_0,$ pick $x_1<x_2$ such that $\int_{x_0}^{x_1}f_0(x) dx = \int_{x_1}^{x_2} f_0(x) dx =F.$ By Lemma \ref{lem.fx_local_bd}, $\int_{x_0}^{x_0+a f_0(x_0)^{1/\beta}/R^{1/\beta}} f_0(x) dx \geq 2F$, which implies
\begin{align}
	x_2\leq x_0+ a f_0(x_0)^{1/\beta}/R^{1/\beta} \leq 1/2+R^{-\frac{1}{\beta(\beta+1)}} n^{-\frac{\beta}{\beta+1}},
	\label{eq.x2x0ineq}
\end{align}
so that $x_2 \leq 1$ for $n\geq n_0(R,\beta)$ large enough.

\begin{figure}
\begin{center}
	\includegraphics[scale=0.7]{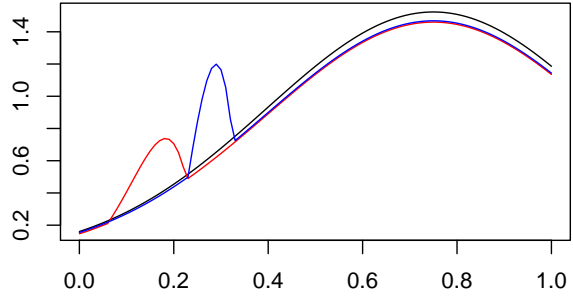}
	\vspace{-1cm}
	\caption{Plot of the densities $f_0$ (black), $f_1$ (red), $f_2$ (blue).\label{fig}}
\end{center}
\end{figure}

Let $K\in \mC^\beta(\mathbb{R})$ be a non-negative function supported on $[0,1]$ and satisfying $\int_0^1 K(u) du =1.$ For $\gamma$ the solution of $\sqrt{1+\gamma} := 1+1/\sqrt{nF},$ consider the two test functions
\begin{align}
	f_j(x) = f_0 (x) \left( 1 -\gamma F + \gamma K\left(\frac{F_0(x)-F_0(x_{j-1})}{F}\right) \right), \quad j\in \{1,2\},
	\label{eq.fj_def}
\end{align}
where $F_0$ is the distribution function of $f_0$. Figure \ref{fig} displays an example of this construction. Since $\gamma = 1/(nF)+2/\sqrt{nF}$, it follows that $\gamma F \leq 3/(2n)$ and $1-\gamma F >0$ for $n\geq 2.$ By substitution, $\int_0^1 f_j(x) dx =1$ and thus the $f_j$ are densities. Moreover, $f_1-f_0$ and $f_2-f_0$ have disjoint support. We also have the following proposition that is proved below.

\begin{prop}
\label{prop.fj_in_mHbeta}
There exists a finite constant $C$, not depending on $R$, such that $f_1, f_2 \in \mH^\beta (CR).$
\end{prop}

Using $\gamma F\leq 3/(2n)$ and $\gamma \leq (1+\sqrt{8/a})^2$ gives 
\begin{align*}
1-\frac{3}{2n}\leq \frac{f_j(x)}{f_0(x)} \leq 1+(1+\sqrt{8/a})^2\|K\|_\infty.
\end{align*}
It therefore follows that $f_1, f_2 \in \{f\in \mH^\beta(cR): c^{-1}f_0 \leq f \leq cf_0\}$ for
\begin{align}\label{eq:c}
c= \max\big(C, 4, 1+(1+\sqrt{8/a})^2\|K\|_\infty\big)
\end{align}
and thus by assumption $f_1,f_2 \in \Theta_n.$ We now establish \eqref{eq.low_bound_to_show} for these $f_1,f_2\in \Theta_n.$

{\it Lower bound for $\tfrac 12 (1-\tfrac 12 \|f_1-f_2\|_1)^n$}: Using $\int_0^1 K(u) du =1$ and substituting $u=(F_0(x)-F_0(x_{j-1}))/F,$ we get $\|f_1-f_2\|_1 = 2\gamma F.$ Since $\gamma F\leq 3/(2n),$
\begin{align}
	\frac 12 \Big( 1- \frac{\|f_{1}-f_{2}\|_{1}}{2}\Big)^n
	\geq \frac 12 \Big( 1- \frac{3}{2n}\Big)^n
	\rightarrow \frac 12 e^{-\frac{3}{2}}.
	\label{eq.lb_part1}
\end{align}

{\it Upper bound for $\Phi(-\sqrt{n}\|\sqrt{f_1}-\sqrt{f_2}\|_2)$}: This is equivalent to lower bounding $ \|\sqrt{f_1}-\sqrt{f_2}\|_2.$ Splitting the integral $\int_0^1$ into $\int_{x_0}^{x_1}+\int_{x_1}^{x_2}+\int_{[x_0,x_2]^c},$ using the properties of $K,$ substitution and the Cauchy-Schwarz inequality yields
\begin{align*}
	 \|\sqrt{f_1}-\sqrt{f_2}\|_2^2 
	  &= 2F\int_0^1 \big( \sqrt{1-\gamma F+\gamma K(u)} - \sqrt{1-\gamma F}\big)^2 du \\
	  &\geq 2F(1-\gamma F) \int_0^1 \big(\sqrt{1+\gamma K(u) } -1\big)^2 du \\
	  &= 2F(1-\gamma F)  \big( \gamma +2 - 2  \int_0^1 \sqrt{1+\gamma K(u) } du \big) \\
	  &\geq 2F(1-\gamma F) \big(\sqrt{1+\gamma}-1 \big)^2 \\
	  &= \frac{2-2\gamma F}n \\
	  &\geq \frac{2}{n} - \frac{3}{n^2}
\end{align*}
where in the last two lines we have used the definition of $\gamma$ and $\gamma  F\leq 3/(2n).$ For $x> 0,$ we find using the standard Gaussian tail bound $\Phi(-x)=1-\Phi(x) \leq (2\pi)^{-1/2}x^{-1} e^{-x^2/2}$, so that we finally obtain
\begin{align*}
	\Phi\big(-\sqrt{n}\big\|\sqrt{f_1}-\sqrt{f_2}\big\|_2\big) 
	\leq \Phi \big(- \sqrt{2}(1+o(1))\big) \rightarrow \Phi \big(-\sqrt{2}\big) \leq \frac{1}{2e\sqrt{\pi}}.
\end{align*}
The last bound and \eqref{eq.lb_part1} together imply \eqref{eq.low_bound_to_show}, which completes the proof of Theorem \ref{thm.lb_small}.\qed

\begin{proof}[Proof of Proposition \ref{prop.fj_in_mHbeta}] In this proof we write $a_n\lesssim b_n$ if $a_n \leq Cb_n$ for a constant $C$ which does not depend on $R$, but might depend on $\beta$ and $K.$ Note that $\gamma \leq (1+\sqrt{8/a})^2.$ Recall that the support of $f_j-f_0$ is $[x_{j-1},x_j]$ and that by Lemma \ref{lem.fx_local_bd}, $\tfrac 12 f_0(x) \leq f_0(x_0) \leq 2f_0(x)$ for all $x\in [x_0,x_2].$ The sup-norm can be easily bounded by $\|f_j\|_{\infty} \leq \|f_0\|_\infty (1+\gamma \|K\|_\infty) \lesssim R.$

For $0<\beta \leq 1,$ using the definition of $F$,
\begin{align*}
	|f_j|_{\mC^\beta}
	&\leq  |f_0|_{\mC^\beta} \big(1+\gamma \|K\|_{\infty}\big)+ 2\gamma f_0(x_0) \left|K\left( \frac{F_0(\cdot)-F_0(x_{j-1})}{F} \right) \right|_{\mC^\beta} \\
	&\leq  R \big(1+\gamma \|K\|_{\infty}\big)+ 2\gamma f_0(x_0) |K|_{\mC^\beta} \sup_{x,y\in[x_{j-1},x_j]:x\neq y} \frac{|F_0(x)-F_0(y)|^{\beta}}{F^\beta |x-y|^\beta} \\
	& \leq R (1+ \gamma \|K\|_\infty) + 2^{\beta+1} \gamma f_0(x_0)^{\beta+1} |K|_{\mC^\beta} F^{-\beta} \\
	&\leq R (1+ \gamma \|K\|_\infty + 16a^{-\beta} \gamma |K|_{\mC^\beta}).
\end{align*}
Since $|f_j|_{\mH^\beta}=0$ for $0<\beta\leq 1$ by definition, it follows that $\|f_j\|_{\mH^\beta} \lesssim R$.

We now bound $|f_j|_{\mC^\beta}$ for $\beta>1.$ Since $|f_0(1-\gamma F)|_{\mC^\beta}\leq R,$ it remains to show $|f_0 \cdot (K \circ v_j)|_{\mC^\beta} \lesssim R$ with $v_j(x):= (F_0(x)-F_0(x_{j-1}))/F.$ Let  $1\leq r \leq \lfloor \beta\rfloor$. For two $r$-times differentiable functions $g, h,$ $(gh)^{(r)} = \sum_{q=0}^r \binom{r}{q}g^{(q)}h^{(r-q)}.$ Moreover, by Fa\`a di Bruno's formula,
\begin{align*}
	\big( K\circ v_j \big)^{(q)} 
	&= \sum \frac{q!}{m_1!\dots m_q!}  (K^{(M_q)}\circ v_j) \prod_{s=1}^q \left( \frac{v_j^{(s)}}{s!} \right)^{m_s} \\
	&= \sum c_{m_1, \ldots, m_q}   \frac{K^{(M_q)}\circ v_j}{F^{M_q}} \prod_{s=1}^q \big(f_0^{(s-1)}\big)^{m_s},
\end{align*}
where the sum is over all non-negative integers $m_1,\ldots,m_q$ with $m_1+2m_2+\ldots+q m_q =q$ and $M_q:=\sum_{s=1}^q m_s$. The $r$-th derivative of $f_0 \cdot (K \circ v_j)$ therefore equals
\begin{align}
	(K\circ v_j) f_0^{(r)}+ \sum_{q=1}^r  \binom{r}{q}  \sum c_{m_1, \ldots, m_q} \frac{ K^{(M_q)}\circ v_j}{F^{M_q}} f_0^{(r-q)} \prod_{s=1}^q \big(f_0^{(s-1)}\big)^{m_s},
	\label{eq.rth_deriv_of_f0Kvj}
\end{align}
where the second sum is over the same set of integers as above. We bound the $|\cdot|_{\mC^\beta}$-seminorm by proving a H\"older bound for each of the terms in \eqref{eq.rth_deriv_of_f0Kvj} individually, starting with the terms in the sum. For $1\leq q \leq r$ and ${\bf m}_q = (m_1,\dots,m_q)$ a $q$-tuple as above, write
\begin{align*}
\varphi(x) = \varphi_{{\bf m}_q} (x) = F^{-M_q} (K^{(M_q)}\circ v_j)(x) f_0^{(r-q)}(x) \prod_{s=1}^q f_0^{(s-1)}(x)^{m_s},
\end{align*}
so that we wish to prove $|\varphi(x)-\varphi(y)| \lesssim R|x-y|^{\beta-\lfloor\beta\rfloor}$. For any $x,y\in[x_0,x_2]$,
\begin{equation}\label{eq.main_split}
\begin{split}
|\varphi(x)-\varphi(y)| & \leq \frac{1}{F^{M_q}} \left| K^{(M_q)}\circ v_j(x) -  K^{(M_q)}\circ v_j(y) \right| \left| f_0^{(r-q)}(x) \prod_{s=1}^q f_0^{(s-1)}(x)^{m_s} \right| \\
& \quad + \frac{1}{F^{M_q}} \left| K^{(M_q)}\circ v_j(y) \left( f_0^{(r-q)}(x) \prod_{s=1}^q f_0^{(s-1)}(x)^{m_s} - f_0^{(r-q)}(y) \prod_{s=1}^q f_0^{(s-1)}(y)^{m_s} \right) \right|.
\end{split}
\end{equation}
Using the definition \eqref{eq.flat_def} of $\mH^\beta$ and that $\sum_{s=1}^q sm_s=q$,
\begin{equation}\label{eq.flat_prod}
\begin{split}
\left| f_0^{(r-q)}(x) \prod_{s=1}^q f_0^{(s-1)}(x)^{m_s} \right|&  \leq R^\frac{r-q}{\beta} f_0(x)^\frac{\beta-(r-q)}{\beta} \prod_{s=1}^q R^\frac{(s-1)m_s}{\beta}f_0(x)^\frac{(\beta-s+1)m_s}{\beta} \\
& = R^\frac{r-M_q}{\beta} f_0(x)^\frac{\beta-r+(\beta+1)M_q}{\beta}.
\end{split}
\end{equation}
Observe that by \eqref{eq.x2x0ineq} and the definition of $F$, for any $1\leq \ell <\beta$ and $x,y \in [x_0,x_2]$,
\begin{align*}
|K^{(\ell)}(v_j(x)) - K^{(\ell)}(v_j(y))| & \lesssim f_0(x_0) F^{-1}|x-y|^{\beta-\lfloor\beta\rfloor} |x_2-x_0|^{1-(\beta-\lfloor\beta\rfloor)} \\
&  \lesssim R^{\frac{\beta -\lfloor \beta\rfloor}{\beta}}f_0(x_0)^{\frac{\lfloor \beta\rfloor-\beta}{\beta}} |x-y|^{\beta-\lfloor \beta \rfloor}.
\end{align*}
Combining the previous two displays and again using the definition of $F$, the first term in \eqref{eq.main_split} is $O(R^\frac{\beta-\lfloor\beta\rfloor+r}{\beta} f_0(x_0)^\frac{\lfloor\beta\rfloor-r}{\beta}|x-y|^{\beta-\lfloor\beta\rfloor})= O(R |x-y|^{\beta-\lfloor\beta\rfloor})$ as required.

For $1\leq \ell<\lfloor \beta\rfloor$, $m \in \mathbb{N}$ and $x,y\in [x_0,x_2]$ with $x<y,$ we have using \eqref{eq.flat_def}, \eqref{eq.x2x0ineq} and the mean value theorem, that for some $\xi=\xi_{x,y}\in [x,y],$
\begin{align*}
|f_0^{(\ell)}(x)^m-f_0^{(\ell)}(y)^m| & \leq m|f_0^{(\ell+1)}(\xi)| |f_0^{(\ell)}(\xi)|^{m-1} (x_2-x_0)^{1-(\beta-\lfloor \beta\rfloor)}|x-y|^{\beta-\lfloor \beta \rfloor} \\
& \lesssim R^{\frac{\beta-\lfloor \beta\rfloor+\ell m}{\beta}} f_0(x_0)^{\frac{\lfloor \beta\rfloor - \beta +(\beta-\ell)m}{\beta}} |x-y|^{\beta-\lfloor \beta \rfloor}.
\end{align*}
Noting that for $\ell=\lfloor\beta\rfloor$, $m_\ell$ only takes values 0 or 1 in the sum in \eqref{eq.rth_deriv_of_f0Kvj}, one can extend the previous display to $\ell=\lfloor\beta\rfloor$ by directly using the H\"older continuity of $f_0^{(\lfloor\beta\rfloor)}$. For the second term in \eqref{eq.main_split}, we repeatedly apply the triangle inequality, each time changing the variable in a single derivative. Fix an integer $1 \leq k\leq q$ and define vectors $(z_s)_{s=1}^q$ and $(\tilde{z}_s)_{s=1}^q$, which are identically equal to $x$ or $y$ in all entries except the $k^{th}$-coordinate, where $z_k=x$ and $\tilde{z}_k=y$. Then using the previous display, a similar argument to \eqref{eq.flat_prod} and the definition of $F$,
\begin{align*}
& \frac{1}{F^{M_q}} \left| f_0^{(r-q)}(x)\right| \left| \prod_{s=1}^q f_0^{(s-1)}(z_s)^{m_s} -  \prod_{s=1}^q f_0^{(s-1)}(\tilde{z}_s)^{m_s}  \right| \\
& = \frac{1}{F^{M_q}} \left| f_0^{(r-q)}(x)\right|  \left| \prod_{s=1,s\neq k}^q f_0^{(s-1)}(z_s)^{m_s} \right|  \left| f_0^{(k-1)}(x)^{m_k} -  f_0^{(k-1)}(y)^{m_k}  \right| \\
& \lesssim R^\frac{\beta-\lfloor\beta\rfloor + r}{\beta} f_0(x_0)^\frac{\lfloor\beta\rfloor-r}{\beta} |x-y|^{\beta-\lfloor\beta\rfloor} \\
&\lesssim R |x-y|^{\beta-\lfloor\beta\rfloor},
\end{align*}
where is the last line we have used that $r\leq \lfloor\beta\rfloor$ and $\|f_0\|_\infty \leq R$. The same inequality can be established for $F^{-M_q}  |f_0^{(r-q)}(x)-f_0^{(r-q)}(y)| |\prod_{s=1}^q f_0^{(s-1)}(y)^{m_s}|$. Since $\|K^{(M_q)}\circ v_j\|_\infty \lesssim 1$, by repeatedly applying the triangle inequality and the last display, the second term in \eqref{eq.main_split} is $O(R |x-y|^{\beta-\lfloor\beta\rfloor})$, so that $|\varphi(x)-\varphi(y)| \lesssim R|x-y|^{\beta-\lfloor\beta\rfloor}$ as required. A similar, but simpler, argument shows that the first term in \eqref{eq.rth_deriv_of_f0Kvj} satisfies $|K\circ v_j(x) f_0^{(r)}(x)-K\circ v_j(y) f_0^{(r)}(y)| \lesssim R|x-y|^{\beta-\lfloor\beta\rfloor}$. This shows that every term in \eqref{eq.rth_deriv_of_f0Kvj}, and hence the whole of \eqref{eq.rth_deriv_of_f0Kvj}, satisfies the required H\"older bound, so that $|f_0 \cdot (K \circ v_j)|_{\mC^\beta} \lesssim R$.

We now prove that $|f_j|_{\mH^\beta} \lesssim R$ for $\beta>1$. By \eqref{eq.flat_def}, it suffices to show $|f_j^{(r)}(x)|\leq (C R)^{\frac r{\beta}} |f_j(x)|^{\frac{\beta-r}{\beta}}$ for all $x\in [0,1]$ and $r=1,\ldots, \lfloor\beta\rfloor$ and a constant $C$ that does not depend on $R.$ Since $K\geq 0$ and $\gamma F \leq 3/(2n),$ it is enough to show that $|f_j^{(r)}(x)|\leq (C' R)^{\frac r{\beta}} |f_0(x)|^{\frac{\beta-r}{\beta}}$ for all $x\in [0,1]$ and a possibly different constant $C'.$ This follows if $|(f_0\gamma (K\circ v_j))^{(r)}(x)|\leq (C'' R)^{\frac{r}{\beta}} |f_0(x)|^{\frac{\beta-r}{\beta}}$ for all $x\in [0,1],$ $r=1,\ldots, \lfloor\beta\rfloor$ and some $C''<\infty.$ This last inequality follows from \eqref{eq.rth_deriv_of_f0Kvj}, \eqref{eq.flat_prod}, the definition of $F$ and that $f_0 \in \mH^\beta(R).$ This also shows that $\|f_j^{(\lfloor \beta \rfloor)}\|_\infty \lesssim R^{\frac{\lfloor \beta\rfloor}{\beta}} \|f_0\|_\infty^{\frac{\beta-\lfloor \beta\rfloor}{\beta}}\lesssim R.$
\end{proof}

\begin{lem}
\label{lem.TV_bd}
For $P$ and $Q$ dominated probability measures, the product measures $P^{\otimes n} = P \otimes \cdots \otimes P$ and $Q^{\otimes n} = Q \otimes \cdots \otimes Q$ satisfy
\begin{align*}
	\big \| P^{\otimes n} - Q^{\otimes n} \big\|_{\TV} \leq 1 - \big(1-  \| P - Q \|_{\TV} \big)^n.
\end{align*}
\end{lem}

\begin{proof}
For probability measures $\widetilde P, \widetilde Q$ on the same measurable space, we have $\|\widetilde P-\widetilde Q\|_{\TV} = 1 - \int d\widetilde P \wedge \widetilde Q.$ If $p,q$ denote the densities of $P$, $Q$ with respect to some dominating measure $\nu$,
\begin{align*}
	\big \| P^{\otimes n} - Q^{\otimes n} \big\|_{\TV}
	= 1-\int \prod_{i=1}^n p (x_i) \wedge  \prod_{i=1}^n q (x_i) d\nu(x_i)
	\leq 1-\prod_{i=1}^n \int  p (x_i) \wedge  q (x_i) d\nu(x_i)
\end{align*}
and the right hand side can be rewritten as $1-(1 - \|P-Q\|_{\TV})^n.$
\end{proof}

\begin{lem}
\label{lem.bds_of_info_distances}
For a function $b$ and $\sigma>0,$ denote by $Q_{b, \sigma}$ the distribution of the path $(Y_t)_{t\in [0,1]}$ with $dY_t= b(t)dt+ \sigma dW_t,$ where $W$ is a Brownian motion. If $\Phi$ denotes the distribution function of the standard normal distribution, then$$\| Q_{b_1, \sigma}- Q_{b_2, \sigma}\|_{\TV} = 1-2\Phi(-\tfrac 1{2\sigma} \|b_1 -b_2\|_2),$$ $$H^2(Q_{b_1, \sigma}, Q_{b_2, \sigma}) = 2-2\exp(-\tfrac 1{8\sigma^{2}} \|b_1 -b_2\|_2^2).$$
\end{lem}

\begin{proof}
This follows from Girsanov's formula $dQ_{b, \sigma}/dQ_{0, \sigma} = \exp(\sigma^{-1}\int_0^1 b(t) dW_t - \tfrac 12 \sigma^{-2}\| b\|_2^2)$ together with $\|P-Q\|_{\TV}= 1-P(\tfrac{dQ}{dP}>1)-Q(\tfrac{dP}{dQ}\geq1)$ and $H^2(P,Q)=2-2\int(dPdQ)^{1/2}.$
\end{proof}

\begin{proof}[Proof of Proposition \ref{prop.parametric_AE}]
We verify the conditions of Theorem 1.2 of Nussbaum \cite{nussbaum1996}, which is a specialized version of the more general Theorem 4.3 of Le Cam \cite{LeCam1985}. We first begin with the regularity conditions stated in Section 10 of \cite{nussbaum1996}. Since $g'(x) =960x(1/2-x)(1-4x)$, for any $\theta\in(0,1/2)$ the Fisher information equals
\begin{align*}
I(\theta) = \int_{\theta}^{\theta+1/2} \frac{g'(x-\theta)^2}{g(x-\theta)} dx = \int_0^1 \frac{g'(x)^2}{g(x)} dx =960\int_0^{1/2} (1-4x)^2dx = 160,
\end{align*}
from which we observe that $I(\theta)$ is both constant and finite for all $\theta\in(0,1/2)$. For $\dot{\ell}_\theta=f_\theta^{-1} \tfrac{\partial}{\partial\theta}f_\theta$, one can show explicitly that for $\theta,\theta+h\in(0,1/2)$, 
\begin{align*}
\int_0^1 \left[\sqrt{f_{\theta+h}}-\sqrt{f_\theta} - \tfrac{1}{2} h \dot{\ell}_\theta\sqrt{f_\theta} \right]^2 = 960h^4,
\end{align*}
so that the family $(f_\theta: \theta\in K)$ is differentiable in quadratic mean uniformly on compact sets $K \subset (0,1/2)$ in the sense of p. 578 of Le Cam \cite{LeCam1986}. Together, these verify the regularity conditions of Proposition 1.2 of \cite{nussbaum1996}, see Section 10 of \cite{nussbaum1996} or Proposition 1, Chapter 17.3 of \cite{LeCam1986}.

We now verify the crucial condition that the family $\Theta(K)$ has finite Hellinger metric dimension. Using the explicit form of $g$ and directly integrating, one can show after some calculations that for any $\theta \in(0,1/2)$,
\begin{align*}
\int_0^1 \sqrt{g(x-\theta)g(x)} dx = \int_\theta^{1/2} \sqrt{g(x-\theta)g(x)} dx = 1-20\theta^2(1-2\theta+8\theta^3/5),
\end{align*}
so that the Hellinger distance equals
\begin{align*}
H^2(P_{f_\theta},P_g) = 2 - 2\int_0^1 \sqrt{g(x-\theta)g(x)} dx = 40\theta^2(1-2\theta+8\theta^3/5).
\end{align*}
By a change of variable, $H(P_{f_\theta},P_{f_{\theta'}}) = H(P_{f_{|\theta-\theta'|}},P_g)$ for any $\theta,\theta'\in(0,1/2)$, so that $H^2(P_{f_\theta},P_{f_{\theta'}}) =40|\theta-\theta'|^2 + O(|\theta-\theta'|^3).$ In particular, there exists a constant $\tilde{c}>1$ such that $\tilde{c}^{-1} |\theta-\theta'|^2 \leq H^2(P_{f_\theta},P_{f_{\theta'}}) \leq \tilde{c}|\theta-\theta'|^2$ for all $\theta,\theta'\in(0,1/2)$. 

To establish finite Hellinger metric dimensionality, we must show that for any $\varepsilon>0$, $\{ f_{\theta'}:H(f_{\theta'},f_{\theta}) \leq \varepsilon\}$ can be covered in Hellinger distance by a finite number of $\varepsilon/2$ balls, independently of $\varepsilon$. By the above results, it suffices to show that $\{ \theta': |\theta'-\theta| \leq \tilde{c}^{1/2} \varepsilon \}$ can be covered by $N$ balls $\{ \theta': |\theta'-\theta_i| \leq \tilde{c}^{-1/2} \varepsilon/2 \}$, $i=1,...,N$, for some $N$ independent of $\varepsilon$. This is simply covering a compact interval in $\R$ and can be done with $N = 2\tilde{c}$ such $\varepsilon/2$ balls, thereby giving the required finite metric dimension $D \leq \log_2(2\tilde{c})$.
\end{proof}

\textbf{Acknowledgements:} The authors would like to thank the referees for their helpful comments and a referee from another paper for suggesting the example in Proposition \ref{prop.parametric_AE}. Most of this work was done while Kolyan Ray was a postdoctoral researcher at Leiden University.

\bibliographystyle{acm}    
\bibliography{bibhd}           

\end{document}